\documentclass{amsart}
\usepackage{color}

\date{ \today}

\title[Decompositions of the automorphism groups of edge-colored graphs]{Decompositions of the automorphism groups of edge-colored graphs into the direct product of permutation groups}
\author{Mariusz Grech}

\keywords{Colored graph, automorphism group, permutation group, direct product.}

\address{Institute of Mathematics, University of Wroc{\l}aw \\
pl.Grunwaldzki 2, 50-384 Wroc{\l}aw, Poland}
\email{Mariusz.Grech@math.uni.wroc.pl}

\begin{document}

\newtheorem{Theorem}{Theorem}[section]
\newtheorem{Observation}[Theorem]{Observation}
\newtheorem{Lemma}[Theorem]{Lemma}
\newtheorem{Fact}[Theorem]{Fact}
\newtheorem{Proposition}[Theorem]{Proposition}
\newtheorem{Corollary}[Theorem]{Corollary}
\newtheorem{Problem}[Theorem]{Problem}
\newtheorem{Example}[Theorem]{Example}

\newcommand{\cnd}  {$\Box$\bigskip}
\newcommand{\V}     {\mbox{$\cal V$}}
\newcommand{\AUT}     {\mbox{$\it AUT$}}
\newcommand{\B}     {\mbox{$ {\cal A}_2 $}}
\newcommand{\A}     {\mbox{${\cal A}_1$}}
\newcommand{\C}     {\mbox{$\cal A$}}
\newcommand{\G}     {\Gamma}
\newcommand{\<}     {\langle }
\renewcommand{\>}   {\rangle }
\newcommand{\s}     {\mbox{$\bf\sigma$}}
\newcommand{\AB}     {\mbox{$\times$}}
\newcommand{\AC}     {\mbox{$\otimes$}}
\newcommand{\rownolegly}[2]{{#1}^{||#2}}
\newcommand{\Au}    {\rm Aut}

\begin{abstract}
In the paper \emph{Graphical complexity of products of permutation groups},
M. Grech, A. Je\.{z}, A. Kisielewicz have proved that the direct product of automorphism groups of edge-colored graphs is itself the automorphism groups of an edge-colored graph.
In this paper, we study the direct product of two permutation groups such that at least one of them fails to be the automorphism group of an edge-colored graph.
We find necessary and sufficient conditions for the direct product to be the automorphism group of an edge-colored graph. The same problem is solved for the edge-colored digraphs.
\end{abstract}

\maketitle

\bigskip

\section{Introduction}

For permutation groups $(A,V), (B,W)$,  the {\it direct product} of $A$ and
$B$ (with product action) is a permutation group $(A \times B, V \times W)$ with the action given by
$$(a,b)(x,y)= (a(x),b(y)).$$

The study of the direct product of automorphism groups of graphs was initiated by G. Sabidussi \cite{sa} in 1960.
The problem was taken up in 1971 by M. Watkins \cite{wa1}.
In 1972, L. Nowitz and M. Watkins \cite{nowa3}, and independently W. Imrich
\cite{imr0}, have described the conditions under which the direct product of \emph{regular} permutation groups that are automorphism groups of graphs is itself the automorphism group of a graph. 
This result was a contribution to the description of all \emph{regular} automorphism groups of graphs completed in 1978 by C. Godsil \cite{god} for graphs and 1980 by L. Babai \cite{ba1} for digraphs.

The results given in \cite{nowa3,imr0} are extended to arbitrary permutation groups in \cite{gre}, where the description of the conditions, under which the direct product of automorphism groups of graphs is itself an automorphism group of a graph, is given. 
In \cite{gre1}, the same is done for digraphs.
In \cite{je}, the direct product of automorphism groups of edge-colored graphs and edge-colored digraphs was studied.
It is shown that the direct product of automorphism groups of edge-colored graphs (digraphs) is, itself, an automorphism group of an edge-colored graph (digraph). 
This and other results in \cite{pei1} and \cite{grekis1} show that the whole problem is more natural for edge-colored graphs(digraphs) than for simple graphs and digraphs.
In \cite{gre5} it has been shown that direct product usually does not require more colors to be represented as the automorphism group of a colored graph than the components themselves.

Considerations on automorphism groups of edge-colored graphs and digraphs have been started by H. Wielandt in \cite{wie}, where permutation groups that are automorphism groups of edge-colored digraphs are called $2$-closed, and those that are automorphism groups of edge-colored graphs are referred to as $2^*$-closed.
In \cite{kis}, A. Kisielewicz has introduced the notion of graphical complexity of permutation groups and suggested the study of products of permutation groups in this context.
By $GR(k)$, we denote the class of automorphism groups of $k$-edge-colored graphs (those using at most $k$ colors), and by $GR$, the union of all the classes $GR(k)$ (that is, $2^*$-closed groups).
Similarly, $DGR(k)$ is the class of automorphism groups of $k$-edge-colored digraphs, and $DGR$ is the union of all the classes $DGR(k)$ (that is, $2$-closed groups). This is clear that $GR
\subseteq DGR$ and $GR(k) \subseteq DGR(k)$, for any $k$.

Now, the main problem is to determine which permutation groups are automorphism groups of edge-colored graphs.
Various aspects of this problem are investigated.
Except for the mentioned above (the list of publications in these direction is much larger), in \cite{grekis2,grekis3}, a description of the so-called totally symmetric graphs is given.
Recently, many other results concerning highly symmetric graphs have been obtained using different terminology of homogeneous factorization of graphs; see the bibliography in \cite{bon,GLPP,LLP,LP}.

A solution of the problem, when the direct product of permutation groups is an automorphism group of an edge-colored graph, may be considered as a contribution to the general problem. In \cite{je}, the following is proved.

\begin{Theorem}
If $A, B \in GR$, then $A \times B \in GR$.  Also, if $A, B \in DGR$, then $A \times B \in DGR$.
\end{Theorem}

The opposite is not generally true.
In this paper, we consider the case when at least one of the components does not belong to $GR$. The main results are (Theorem~\ref{GR}) in which conditions under which the direct product of two permutation groups belongs to $GR$ are given, and
Theorem~\ref{NGR} showing that in case of digraphs the condition that both the components are in DGR is also necessary.

\section{Preliminaries}

We assume that the reader has basic knowledge in the areas of graphs and permutation groups, so we omit an introduction to standard terminology. If necessary, additional details can be found in \cite{ba,yap}.

By a {\it $k$-edge-colored graph} $G$, we mean a pair $G = (V,E)$, where
$V$ is the set of vertices of $G$, and $E$ the {\it edge-color function} from the set $P_2(V)$ of unordered pairs of vertices into the set of colors
$\{ 0, \ldots, k-1\}$ ($E : P_2(V) \rightarrow \{ 0, \ldots, k-1\}$). Thus,
$G$ is a complete simple graph with colored edges. Similarly, by a {\it $k$-edge-colored digraph} $G$, we mean a pair $(V,E)$ where $E$ is a color function from the set of ordered pairs of different elements of
$V$ to the set of colors $\{ 0, \ldots, k~-~1\}$ ($E: ((V \times V) \setminus \{(v,v); v \in V\}) \rightarrow \{ 0, \ldots, k-1\}$).

An automorphism of an edge-colored graph $G$ is a permutation $\sigma$ of the set $V$ preserving the edge function: $E(\{ v, w\}) = E( \{ \sigma(v), \sigma(w) \})$, for all $v,w \in V$.
The group of automorphisms of $G$ will be denoted by $Aut(G)$, and considered as a permutation group $(Aut(G),V)$ acting on the set of the vertices $V$.
The similar definitions are for edge-colored digraphs.

All groups considered in this paper are groups of permutations.
They are considered up to permutation group isomorphism.
Generally, a permutation group $A$ acting on a set $V$ is denoted $(A,V)$ or just $A$, if the set $V$ is clear from the context or not important. By
$S_n$ we denote the symmetric group on $n$ elements, and by $I_n$, the one element group acting on $n$ elements (consisting of the identity only, denoted by $id$).
$E(\{v_i,v_{(i+1\; \mod n)}\})=1$ for all $i$, and $E(v_i,v_j)=0$, otherwise.

With a given group of permutations $(A,V)$ we associate two other permutation groups $A_1$ and $A_2$, abstractly  isomorphic with $A$. The group $A_1$ acts on the set $P_2(V)$ of the unordered pairs of $V$ in the following way.
If $a \in A$ and $\sigma: A \to A_1$ is an abstract automorphism, then
$\sigma(a)(\{v,w\})=\{a(v),a(w)\}$.
The group $A_2$ acts on the set  $(V \times V) \setminus \{(v,v); v \in V\}$ in the following way.
If $a \in A$ and $\tau: A \to A_2$ is an abstract automorphism, then
$\tau(a)((v,w))=(a(v),a(w))$.
Later on, we will identify $A$, $A_1$, and $A_2$, and we will be writing that $A$ acts on $P_2(V)$ or $A$ acts on $(V \times V) \setminus \{(v,v); v \in V\}$ in the meaning as above.
The orbits of $A$ in the action on $P_2(V)$ are known as orbitals of $A$.
However, to distinguish them from the other we will call them
\emph{NOr-orbitals} of $A$, and the orbits of $A$ in the action on $(V
\times V) \setminus \{(v,v); v \in V\}$, we will call \emph{Or-orbitals} of $A$.

For two Or-orbitals $O_1, O_2$ we say that $O_1$ is {\it paired}  with
$O_2$ if and only if $O_2 = \{(w,v):  (v,w) \in O_1\}$.
We call an Or-orbital $O$ {\it self-paired} if it is paired with itself.
Moreover, we say that a permutation $\s$ pairs $O_1$ and $O_2$ if $\s(O_1)
= O_2$ and $O_1$ is paired with $O_2\ne O_1$.

Since $A \times I_1 = I_1 \times A = A$, in this paper, we consider only the direct products $A \times B$ with both the permutation groups $A, B$ different from $I_1$.

Let $A$ be a permutation group. By $\bar{A}$, we denote the smallest permutation group (that acts on the same set) which contains $A$ and belongs to $GR$.
This is clear that the group $\bar{A}$ exists.
Moreover, we have $A \in GR$ if and only if $A= \bar{A}$.
Let $O_0, \ldots O_{k-1}$ be all the NOr-orbitals of $A$.
For a group $(A,V)$, we define an edge-colored graph $G(A)$ as follows.
$$G(A) = (V,E), \textrm{ where } E :P_2(V) \rightarrow \{ 0, \ldots k-1\}.$$
$$E(\{ v,w \})= i \textrm{\,\, if and only if the edge } \{ v,w\} \textrm{ belongs to the orbit } O_i.$$

\vspace{3mm}

\begin{Fact}\label{f1}
$Aut(G(A)) = \bar{A}$.
\end{Fact}
\begin{proof}
Choose an arbitrary edge $e$ of $G(A)$ and $a \in A$.
There is $i$ such that $e \in O_i$.
Since $O_i$ is an orbital of $A$, we have $a(e) \in O_i$.
Hence, by definition of $G(A)$ the edges $e$ and $a(e)$ have the same color.
It follows that $Aut(G(A)) \supseteq A$.
Consequently, since $\bar{A}$ is the smallest group in $GR$ that contains
$A$, we have  $Aut(G(A)) \supseteq \bar{A}$.

For the opposite inclusion take an arbitrary $B \in GR$ such that $B \supseteq A$.
Then, $B = Aut(H)$ for some $H=(V,F)$.
Let $e_1, e_2$ be two edges that belong to the same orbital.
Since there is $a \in A$ that $a(e_1) = a(e_2)$ and $A \subseteq B$, we have $F(e_1) = F(e_2)$.
Let now $b \in Aut(G(A))$ and $e$ be an arbitrary edge of a graph $H$.
Then, $E(a(e))=E(e)$.
As we have just shown $F(a(e))=F(e)$.
Hence, $b \in B$.
Thus, $Aut(G(A)) \subseteq B$.
Consequently, $Aut(G(A)) \subseteq \bar{A}$.
Therefore, $Aut(G)=\bar{A}$
\end{proof}

\begin{Fact}\label{f2}
 $A \times B \subseteq Aut(G(A \times B)) \subseteq \bar{A} \times \bar{B}$,
\end{Fact}

\begin{proof}
Exactly in the same way as in proof of the Fact~\ref{f1} above we have  $A \times B \subseteq Aut(G(A \times B))$.
We prove the second inclusion.

Observe that the edges of the form $\{(v_1,w),(v_2,w)\}$ (belonging to the rows)  belong to the other NOr-orbitals than the edge
$\{(v_1,w_1),(v_2,w_2)\}$ with $w_1 \ne w_2$.
Therefore, the edges in rows have the other colors than the rests of the edges.
The same is true for columns.
Thus, rows can be mapped only onto rows and columns can be mapped only onto columns.
This implies that $Aut(G(A \times B)) \subseteq A_1 \times B_1$, for some $A_1$ and $B_1$.
Let now $(a,b) \in Aut(G(A \times B)))\}$.
Then, $(a,b)(\{(v_1,w),(v_2,w)\})$ and $\{(v_1,w),(v_2,w)\}$ have the same color.
Therefore, there is $(a_1,b_1) \in A \times B$ such that
$(a_1,b_1)(\{(v_1,w),(v_2,w)\})=\{(v_1,w),(v_2,w)\}$.
Hence, $(a_1^{-1}a,b_1^{-1}b) \in Aut(G(A \times B))$ preserves the row with the edge $\{(v_1,w),(v_2,w)\}$.
Since every row is a copy of $G(A)$ (with other names of the colors), we have $a_1^{-1}a \in \bar{A}$ which implies that $a \in \bar{A}$.
In the similar way $b \in \bar{B}$
\end{proof}

\section{Results.}\label{gr1}

The aim of this section is to describe all the cases, where the condition
$A \times B = Aut(G(A \times B))$ holds.

We denote $A'=\bar{A} \setminus A$.
Let $a \in A'$.
Obviously, $a$ preserves $NOr$-orbitals of $A$.
We show the following lemma that will be used in a sequel.

\begin{Lemma}\label{or}
Let $A \notin (GR \cup \{I_2\})$.
If $a \in A'$, then $a$ preserves the orbits of $A$.
\end{Lemma}

\begin{proof}
Let $O_t$, $t \in \{1, \ldots, m\}$.
The lemma is nontrivial only for $m >1$
We consider two cases.
First, we assume that there is a nontrivial orbit of $A$.
We may assume that this is the orbit $O_1$.
Then, the set $P_2(O_1)$ is nonempty.
Moreover, it is clear that $P_2(O_1)$ is the sum of some number of orbitals of $A$.
Hence, the edges that belongs to $P_2(O_1)$ have different colors than the rest of the edges.
This implies that $a$ preserves the orbit $O_t$.
Moreover, it is also clear that for a fixed $t \in \{2, \ldots m\}$, the edges $\{v,w\}$, where $\{v,v\} \in O_1$ and $\{w,w\} \in O_t$ also have different colors than the rest of the edges.
Consequently, every orbit is preserved by $a$.

In the case, when all the orbits $O_t$ are trivial, we have $A =I_{|V|}$ and $|V|>2$. Obviously, in this case, $a(O_t)=O_t$.
\end{proof}

Now, we make a natural observation, which is an immediately implication of the definition, but it is crucial for this paper.

\begin{Observation}
Let $A$ be a permutation group. Then, $A \notin DGR$ if and only if there is $a \notin A$ such that $a$ preserves all Or-orbitals of $A$. Moreover, in such a case, $a$ belongs to $A'$.
\end{Observation}

At least we  are ready to prove the first theorem.

\begin{Theorem} \label{NGR}
Let $A \notin DGR$ and $B$ be an arbitrary permutation group.
Then, $A \times B \notin DGR$.
\end{Theorem}

\begin{proof}
Let $A = (A,V)$ and $B=(B,W)$.
Since $I_2 \in DGR$, $A \ne I_2$.
We define an edge-colored digraph $Gr(A \times B)$ in the similar way as $G(A \times B)$.
The only difference is that we color Or-orbits not NOr-orbits of the group $A \times B$.
This implies that $A \times B \subseteq Aut(Gr(A \times B)) \subseteq Aut(G(A \times B))$.
In the same way as for graphs we have that $A \times B \in DGR$ if and only if $A \times B = Aut(Gr(A \times B))$. (Follow Fact~\ref{f1} and Fact~\ref{f2}.)

Choose $a \in A'$ which preserves all Or-orbitals of $A$. 
Obviously, $a \in Aut(Gr(A \times B)$.
Let $id_B$ be the identity in the permutation group $B$.
We show that the permutation $(a,id_B)$ belongs to $Aut(Gr(A \times B))$.
To this end, we show that for every directed edge $e = ( (v_1,w_1),
(v_2,w_2) )$, where $v_1,v_2 \in V$, $w_1,w_2 \in W$,  the image
$(a,id_B)(e)$ has the same color as $e$.

Assume first $v_1 \ne v_2$.
We know that $a$ preserves every Or-orbital of $A$.
Thus, for every pair $(v_1,v_2)$, there is a permutation $a_2 \in A$ such that $a(v_1)=a_2(v_1)$ and $a(v_2)=a_2(v_2)$.
We have $(a,id_B)(e)=(a_2,id_B)(e)$, and therefore the directed edges
$(a,id_B)(e)$ and $e$ belong to the same Or-orbital of $A \times B$.
So, by the definition of the edge-colored digraph $Gr(A \times B)$,
$(a,id_B)(e)$ and $e$ have the same color.

If $v_1 = v_2$, since $A \ne I_2$, we may use Lemma~\ref{or} and find a permutation $a_1 \in A$ such that $a_1(v_1)=a(v_1)$.
We have $(a,id_B)(e)=(a_1,id_B)(e)$, and therefore the directed edges
$(a,id_B)(e)$ and $e$ belong to the same Or-orbital of $A \times B$.
So, they have the same color.

Thus, in all the cases $(a,id_B) \in Aut(Gr(A \times B))$ but $(a,id_B)$ does not belong to $A \times B$.
Therefore, $A \times B \notin DGR$.
\end{proof}

\begin{Corollary}\label{DGR}
$A \times B \in DGR$ if and only if $A, B \in DGR$.
\end{Corollary}

This completes the problem of direct product in the edge-colored digraphs case.
However, in the case of edge-colored graphs, we have still a few cases to consider.
We write an obvious consequence of Theorem~\ref{NGR}.

\begin{Corollary}\label{GRN}
Let $A \notin DGR$ and $B$ be an arbitrary permutation group.
Then, $A \times B \notin GR$.
\end{Corollary}

According to the Corollary~\ref{GRN}, we will assume further that every considered group belongs to $DGR$.
First, we consider a few cases where $A \in DGR \setminus (GR \cup
\{I_2\})$ and $B \in GR$.

\begin{Lemma} \label{self-paired}
Let $A \in DGR \setminus (GR \cup \{I_2\})$ and $B \in GR$.
If every Or-orbital of $B$ is self-paired, then $A \times B \not\in GR$.
\end{Lemma}

\begin{proof}
Let $A=(A,V)$ and $B=(B,W)$. We fix an element $a \in A'$.
Let $id_B$ be the identity in the permutation group $B$, and $e=\{
(v_1,w_1),(v_2,w_2) \}$, where
$v_1,v_2 \in V$, $w_1,w_2 \in W$.
We show that the edges $e$ and $(a,id_B)(e)$ have the same color.
To this end it is enough to prove that $(a,id_B)(e)$ belongs to the same
NOr-orbital of $A \times B$ as $e$ does.

If $w_1 = w_2$, then the statement holds according to the fact that $a$ preserves all NOr-orbitals of $A$.
Assume $v_1 = v_2$. Since $A \ne I_2$, by Lemma~\ref{or}, $a$ preserves all orbits of $A$ (in action on $V$).
Hence, there is $a_1 \in A$ such that $a(v_1)=a_1(v_1)$.
We have,
\begin{eqnarray}
\nonumber &(a,id_B)(\{(v_1,w_1),(v_1,w_2)\})=\{(a(v_1),w_1),(a(v_1),w_2)\}=
&\\
\nonumber & (a_1,id_B)(\{(v_1,w_1),(v_1,w_2)\}). &
\end{eqnarray}
Thus, $e$ and $(a,id_B)(e)$ belong to the same NOr-orbital of $A \times B$.

Let now $v_1 \ne v_2$ and $w_1 \ne w_2$.
If the pair $a((v_1,v_2))$ belongs to the same Or-orbital of $A$ as the pair $(v_1,v_2)$, then there is $a_1 \in A$ such that $a_1(v_1) = a(v_1)$ and $a_1(v_2) = a(v_2)$.
Similarly as above, we have,
\begin{eqnarray}
\nonumber &(a,id_B)(\{(v_1,w_1),(v_2,w_2)\})=\{(a(v_1),w_1),(a(v_2),w_2)\}=
&\\
\nonumber & (a_1,id_B)(\{(v_1,w_1),(v_2,w_2)\}). &
\end{eqnarray}

Assume, finally, that $v_1 \ne v_2$, $w_1 \ne w_2$ and the pairs
$a((v_1,v_2))$, $(v_1,v_2)$ belong to the different Or-orbitals of $A$.
Since $a \in \bar{A}$, we know that $a$ preserves all Nor-orbitals of $A$.
This implies that, the pairs $a((v_1,v_2))$ and $(v_2,v_1)$ belong to the same Or-orbital of $A$.
Hence, there is $a_1 \in A$ such that $a_1((v_2,v_1))=a((v_1,v_2))$.
Moreover, since all Or-orbitals of $B$ are self-paired, there is $b \in B$ such that $b((w_1,w_2))=(w_2,w_1)$.
Consequently,
$$(a,id_B)(e)= \{ (a_1(v_2),b(w_2)), (a_1(v_1),b(v_1))\} = (a_1,b)(e).$$
Thus $(a,id_B)(e)$ and $e$ belongs to the same NOr-orbital of $A \times B$.
Thus $(a,id_B)$ does not change the color of the edges.
Therefore, $(a,id_B) \in Aut(G(A \times B))$.
Thus, $(a,id_B) \in \overline{A \times B}$.
Since $a \in A'$, $(a,id_B) \notin A \times B$. 
Therefore, $A \times B \ne
\overline{A \times B}$ and $A \times B \notin GR$.
This completes the proof.
\end{proof}

The examples of groups with all Or-orbitals self-paired are $S_n$ and their transitive products (direct, wreath, ...), in particular all groups $S_2^p$.
Other examples are totally symmetric groups that are described in
\cite{grekis2}.
In fact every such a group of permutations $A$ has the property that if $A
=Aut(G)$ for a digraph $G$, then $G$ is undirected (simple) graph.

A different situation is when a permutation group $B \in GR$ has some non-self-paired Or-orbital and $A \in DGR \setminus (GR \cup \{I_2\})$.

\begin{Lemma} \label{not-self-paired}
Let $A \in DGR \setminus (GR \cup \{I_2\})$ and $B \in GR$ have at least one not self-paired Or-orbital.  
Then, $A \times B \in GR$.
\end{Lemma}

\begin{proof}
Let $A=(A,V)$ and $B=(B,W)$.
We know that $Aut(G(A \times B)) \subseteq \bar{A} \times B$.
Therefore, every $\s \in Aut(G(A \times B))$ has a form $(a,b)$, where $a
\in \bar{A}$ and $b \in B$.
We show that, in fact, $a$ always belongs to $A$.
Assume, to the contrary, that $a \in A'$. In this case, since $A \in DGR
\setminus (GR \cup \{I_2\})$, there is (an ordered) pair $(v_1,v_2),
v_1,v_2 \in V$ such that $a((v_1,v_2)) \not= a_1((v_1,v_2))$, for every $a_1 \in A$.
Since $B$ has an Or-orbital which is not self-paired, there are $w_1, w_2
\in W$ such that $b((w_1,w_2)) \not= (w_2,w_1)$ for every $b \in B$.
Now, observe that the edges $(a,b)(\{ (v_1,w_1),(v_2,w_2)\})$ and $\{
(v_1,w_1),(v_2,w_2)\}$ belong to different NOr-orbitals of $A \times B$.
Indeed, if the edges $(a,b)(\{ (v_1,w_1),(v_2,w_2)\})$ and $\{
(v_1,w_1),(v_2,w_2)\}$ belong to the same NOr-orbital of $A \times B$, then either there are $a_1 \in A$ and $b_1 \in B$ such that
$a((v_1,v_2))=a_1((v_1,v_2))$ and $b((w_1,w_2))=b_1((w_1,w_2))$ or there are $a_2 \in A$ and $b_2 \in B$ such that $a((v_1,v_2))=a_2((v_2,v_1))$ and
$b((w_1,w_2))=b_2((w_2,w_1))$. The former case is impossible by assumption on $a$, the other, since it would imply $b_2^{-1}b((w_1,w_2))=(w_2,w_1)$, which was forbidden.
This implies that $E( (a,b)(\{ (v_1,w_1),(v_2,w_2)\} )) \not=$ $E( \{
(v_1,w_1),(v_2,w_2)\} )$, which contradicts the fact that $(a,b) \in Aut(G(A \times B))$.
Consequently, we have $Aut(G(A \times B)) \subseteq A \times B$.
This completes the proof.
\end{proof}

We summarize Lemma \ref{self-paired} and Lemma \ref{not-self-paired}.

\begin{Corollary} \label{Nie-Tak}
Let $A \in  DGR \setminus (GR \cup \{I_2\})$ and $B \in GR$. 
Then, $A \times B \in GR$ if and only if there exists a non-self-paired Or-orbital of $B$. \hfill \cnd
\end{Corollary}

Now, we have to consider one special case.

\begin{Theorem}\label{I_2}
Let $B \in (GR \cup \{I_2\})$. 
Then, $B \times I_2 \in GR$.
\end{Theorem}

\begin{proof}
It is known, and easy to check, that $I_4 = I_2 \times I_2 \in GR(3)$.
Thus assume that $B \in GR$.
We know that $Aut(G(B \times I_2))$ is either $B \times I_2$ or $B \times S_2$.
By our general assumption $B \ne I_1$, hence there is at least one edge of
the form $\{(v,0),(w,0)\}$ and has different color then $\{(v,1),(w,1)\}$.
Thus, the latter is excluded.
\end{proof}

This completes the description in all the cases where at least one of the components belongs to $GR$.
However, for some subclasses of $GR$ we may state the result in a little nicer form.
Since all intransitive permutation groups have a non-self-paired Or-orbit, we have the following.

\begin{Corollary}
Let $A \in (DGR \setminus GR)$, and $B \in GR$ be intransitive. 
Then, $A \times B \in GR$. \hfill $\Box$
\end{Corollary}

Moreover, it is easy to observe that the only regular groups with all self-paired Or-orbitals are $S_2^n, n \geq 1$.
This implies that:

\begin{Corollary}
Let $A \in (DGR \setminus GR)$, and $B \in GR$ be regular. Then,
$A \times B \in GR$ if and only if $B \not= S_2^n$, for every $n$. \hfill
$\Box$
\end{Corollary}

The remaining case occurs where $A,B \in (DGR \setminus GR)$. We start with the following.

\begin{Lemma} \label{stabilizowane-orbity}
Let $A,B \in (DGR \setminus GR)$.
If for every $b \in B'$ there exists a pair of paired Or-orbitals $O_1 \ne
O_2$ of $B$ such that $b$ does not pairs $O_1$ and $O_2$, then $A \times B \in GR$.
\end{Lemma}

\begin{proof}
Let $A=(A,V)$ and $B=(B,W)$.
Assume to the contrary that there exists $(a,b) \in Aut(G(A\times B))
{\setminus} (A \times B)$.

First, assume that $a \in A$; then, $b \notin B$. Since $A \in (DGR
\setminus GR)$, there is (an ordered) pair $(v_1,v_2)$, where $v_1,v_2 \in
V$, which belongs to a non-self paired Or-orbital of $A$. Since $B \in
DGR$, there is (an ordered) pair $(w_1,w_2)$ where $w_1,w_2 \in W$, for which there is no $b_1 \in B$ such that $b_1((w_1,w_2))=b((w_1,w_2))$.
We prove that the edge $\{(v_1,w_1),(v_2,w_2)\}$ belongs to a different
 NOr-orbital than the edge $(a,b)(\{(v_1,w_1),(v_2,w_2)\})$.
Indeed, if the edges $(a,b)(\{(v_1,w_1),(v_2,w_2)\})$ and
$\{(v_1,w_1),(v_2,w_2)\}$ belong to the same Nor-orbit, then either there are $a_1 \in A$ and $b_1 \in B$ such that $a((v_1,v_2)) = a_1((v_1,v_2))$ and $b((w_1,w_2))=b_1((w_1,w_2))$ or there are $a_2 \in A$ and $b_2 \in B$ such that $a((v_1,v_2)) = a_2((v_2,v_1))$ and
$b((w_1,w_2))=b_2((w_2,w_1))$.
In the former, by assumption on $b$ and $w_1,w_2$, this is impossible.
In the other, since $a \in A$ it is also impossible.
Hence, the edges $(a,b)(\{(v_1,w_1),(v_2,w_2)\})$ and
$\{(v_1,w_1),(v_2,w_2)\}$ have different colors.
This contradicts the assumption that $(a,b) \in Aut(G(A\times B))$.

We consider the case where $a \notin A$. Since $A \in DGR$, there is an ordered pair $(v_1,v_2)$, where $v_1,v_2 \in V$, for which there is no permutation $a_1 \in A$ such that $a_1((v_1,v_2))=a((v_1,v_2))$.
By assumption, there is an ordered pair $(w_1,w_2)$, $w_1,w_2 \in W$ such that $(w_1,w_2)$ belongs to not self-paired Or-orbital and $b((w_1,w_2)) = b_1((w_1,w_2))$ for some $b_1 \in B$.
A similar proof as above shows that the edge
$$(a,b)(\{(v_1,w_1),(v_2,w_2)\})=(a,b_1)(\{(v_1,w_1),(v_2,w_2)\})$$ belongs to a different NOr-orbital than the edge $\{(v_1,w_1),(v_2,w_2)\}$. 
Again, this contradicts the assumption that $(a,b) \in Aut(G(A\times B))$.
\end{proof}

Now, we consider the case where one of the groups is equal to $I_2$.

\begin{Lemma}
Let $A \in (DGR \setminus GR)$. 
Then, $A \times I_2 \in GR$.
\end{Lemma}

\begin{proof}
Let $A=(A,V)$ and $I_2=(I_2,\{w_1,w_2\})$.
Assume to the contrary that there is $(a,b) \in Aut(G(A \times I_2)) \setminus (A \times I_2)$.
Since, for any $v_1, v_2, v_3, v_4 \in V$, the edges
$\{(v_1,w_1),(v_2,w_1)\}$ and $\{(v_3,w_2),(v_4,w_2)\}$ have different colors, $b = id$.
In the same way as in the second case of the proof of the previous lemma, we get a contradiction.
\end{proof}

Now, we consider the last case.

\begin{Lemma} \label{2}
Let $A,B \in DGR \setminus (GR \cup I_2)$.
If there exists $a \in A'$ which pairs all the pairs of the paired
Or-orbitals of $A$ and there exists $b \in B'$ which pairs all the pairs of the paired Or-orbitals of $B$, then $A \times B \not\in GR$. Moreover, $A
\times B$ is transitive.
\end{Lemma}

\begin{proof}
Let $A=(A,V)$ and $B=(B,W)$. Since $A \ne I_2$ and $B \ne I_2$, by
Lemma~\ref{or}, every permutation $a \in A'$ preserves the orbits of $A$ (in action on $V$) and every permutation $b \in B'$ preserves the orbits of $B$ (in action on $W$).
Hence, we obtain immediately, under the assumptions on $A$ and $B$, that the permutation groups $A$ and $B$ have to be transitive. Consequently, for every $a \in \bar{A}, b \in \bar{B}$, $v,v_1,v_2 \in V$, and $w, w_1, w_2
\in W$, the edge $(a,b)(\{(v,w_1),(v,w_2)\})$ has the same color as the edge $(\{(v,w_1),(v,w_2)\})$, and moreover, the edge
$(a,b)(\{(v_1,w),(v_2,w)\})$ has the same color as the edge $(\{(v_1,w),(v_1,w)\})$.

We choose $a$ and $b$ as in the formulation of the lemma, and fix the elements $v_1, v_2 \in V$ and $w_1, w_2 \in W$.
Since $a$ and $b$ preserves no non-self-paired Or-orbit, the ordered pair
$a((v_1,v_2))$ belongs to the Or-orbital of the ordered pair $(v_2,v_1)$ and the ordered pair $b((w_1,w_2))$ belongs to the Or-orbital of the ordered pair $(w_2,w_1)$.
Hence, there are $a_1 \in A$ and  $b_1 \in B$ such that $a((v_1, v_2))=a_1((v_2,v_1))$ and $b((w_1,w_2))=b_1((w_2,w_1))$.
Therefore, we have
\begin{eqnarray}
\nonumber & E( (a,b) (\{ (v_1,w_1), (v_2,w_2) \}) ) = E(  \{
(a(v_1),b(w_1)), (a(v_2),b(w_2)) \} ) = & \\
\nonumber & E(  \{ (a_1(v_2),b_1(w_2)), (a_1(v_1),b_1(w_1)) \} ) = E(
(a_1,b_1)( \{ (v_1,w_1), (v_2,w_2) \}) ) = & \\
\nonumber & E(  \{ (v_1,w_1), (v_2,w_2) \} ). & \end{eqnarray}
The vertices $v_1$, $v_2$, $w_1$, and $w_2$ are arbitrary.
Hence, the permutation $(a,b)$ preserves all colors. 
Consequently, $(a,b) \in Aut(G(A \times B)\setminus (A \times B))$.
\end{proof}

Summarizing this section, we have:

\begin{Theorem} \label{GR}
Let $A$ and $B$ be permutation groups.
Then, $A \times B \in GR$, except for the following cases:

\begin{enumerate}

\item[(i)] $A \times B \notin DGR$, i.e., either $A \notin DGR$ or $B
\notin DGR$,

\item[(ii)] either every Or-orbital of $A \in GR$ is self-paired and $B
\notin GR \cup \{I_2\}$ or every Or-orbital of $B \in GR$ is self-paired and $A \notin GR \cup
\{I_2\}$,

\item[(iii)] $A,B \in DGR \setminus (GR \cup \{I_2\})$ are transitive with at least one element in $A'$ or $B'$, respectively, which pairs all the pairs of the paired Or-orbitals.
\end{enumerate}

\end{Theorem}

\section{Corollaries and problems}

\begin{Problem}
Describe all the groups with the property as in {\rm Theorem~\ref{GR}
(iii)}
\end{Problem}

The groups with this property are all regular abelian and regular generalized dicyclic groups.
However there are also many others transitive permutation groups with this property.

\begin{Example}
$A=\langle (0,1,2,3,4,5,6), (1,2,4)(3,6,5)\rangle$ has this property.
\end{Example}

This is a subgroup of $F_7$ generated by translations and multiplication by
$2$ that has order $3=6/2$.

\begin{Problem}
If any subgroup of $F_{p^n}$ generated by translations and $\omega^{2k}$, where $\omega$ is a generator of the multiplicative group $F^*_{p^n}$, and
$k$ divides $n$, has this property?
\end{Problem}

\begin{Corollary}
Except for the abelian groups of exponent greater than two and generalized dicyclic groups, all the finite regular permutation groups belong to the class $GR$.
\end{Corollary}

\begin{proof}
Let $A$ be an abelian group of exponent greater than two or a generalized dicyclic group.
Then, it is known (see, \cite{god}, for instance) that $A \notin GR(2)$.
In this same way, $A \notin GR$.
Assume that $A$ is not of this form.
Then, it is well known (see \cite{god}) that $A \times S_2^4 \in GR(2)$.
Since $S_2^4 \in GR$ and it has all Or-orbitals self-paired, then by
Theorem~\ref{GR}~(ii), $A \in GR$.
\end{proof}

As another corollary, we have another proof of a well-known fact.

\begin{Corollary}
Every regular permutation group belongs to $DGR$.
\end{Corollary}

\begin{proof}
Let $S$ be an unsolvable regular group.
Then, for every regular group $A$, the group $A \times S$ is unsolvable.
By \cite{god}, we have $A \times S \in GR(2) \subseteq DGR$.
By Corollary~\ref{DGR}, $A \in DGR$.
\end{proof}

\end{document}